\documentclass{amsart}

\usepackage{amssymb}
\usepackage{amsmath}
\usepackage{tikz}
\usetikzlibrary{decorations.markings}

\newtheorem{thm}{Theorem}[section]
\newtheorem{prop}[thm]{Proposition}
\newtheorem{lem}[thm]{Lemma}
\newtheorem{cor}[thm]{Corollary}

\theoremstyle{definition}
\newtheorem{definition}[thm]{Definition}
\newtheorem{example}[thm]{Example}

\theoremstyle{remark}

\newtheorem{remark}[thm]{Remark}

\numberwithin{equation}{section}

\DeclareMathOperator{\area}{Area}
\DeclareMathOperator{\lk}{lk}

\begin{document}

\title[Strictly systolic angled complexes and hyperbolicity]{Strictly systolic angled complexes and hyperbolicity of one-relator groups}

\author{Mart\'in Axel Blufstein}
\author{Elias Gabriel Minian}
\address{Departamento  de Matem\'atica - IMAS\\
 FCEyN, Universidad de Buenos Aires. Buenos Aires, Argentina.}
\email{mblufstein@dm.uba.ar ; gminian@dm.uba.ar}

\thanks{Researchers of CONICET. Partially supported by grant UBACyT 20020160100081BA}



\begin{abstract}
We introduce the notion of strictly systolic angled complexes. They generalize Januszkiewickz and \'Swi\c{a}tkowski's $7$-systolic simplicial complexes and also their metric counterparts, which appear as natural analogues to Huang and Osajda's metrically systolic simplicial complexes in the context of negative curvature. We prove that strictly systolic angled complexes and the groups that act on them geometrically, together with their finitely presented subgroups, are hyperbolic. We use these complexes to study the geometry of one-relator groups without torsion, and prove hyperbolicity of such groups under a metric small cancellation hypothesis, weaker than $C'(1/6)$ and $C'(1/4)-T(4)$.
\end{abstract}

\maketitle

\section{Introduction}

A strictly systolic angled complex is a simplicial complex in which we admit multiple edges between vertices, such that the $2$-skeleton is a nonnegative angled $2$-complex in the sense of Wise \cite{Wi} and with a link condition similar to Huang and Osajda's $2\pi$-large condition for metrically systolic simplicial complexes \cite{HO}. This new notion is flexible enough to include objects of combinatorial nature, such as Januszkiewickz and \'Swi\c{a}tkowski's $7$-systolic simplicial complexes \cite{JS}, and also of geometric nature, such as a variation, for negative curvature, of Huang and Osajda's metrically systolic simplicial complexes. We will show that these complexes satisfy a linear isoperimetric inequality, combining ideas of the weight tests of Gersten \cite{Ger} and Huck and Rosebrock \cite{HR}, with reduction methods for diagrams similar to those of Januszkiewickz and \'Swi\c{a}tkowski \cite{JS}. In particular, groups acting properly and cocompactly by simplicial automorphisms on these complexes are hyperbolic. Moreover, by a result of Hanlon and Mart\'inez-Pedroza \cite[Theorem 1.1]{HP} that is based on equivariant towers, the finitely presented subgroups of such groups are also hyperbolic.

We use strictly systolic angled complexes to investigate the geometry of one-relator groups. It is well known that all one-relator groups with torsion are hyperbolic. This is an immediate corollary of Newman's Spelling Theorem \cite{New}. Recall that a one-relator group has torsion if and only if its relation is a proper power (see \cite{KMS}). On the other hand, the geometry of one-relator groups without torsion is more intricated. Ivanov and Schupp described hyperbolicity in some classes of one-relator groups \cite{IS}. Recently Gardam and Woodhouse exhibited a family of one-relator groups $R_{pq}$ which satisfy a polynomial isoperimetric inequality. They showed that the groups $R_{pq}$ have no subgroups isomorphic to a Baumslag-Solitar group $BS(m,n)$ (with $m\neq n,-n$) and they are neither automatic nor $CAT(0)$ \cite{GW}.

It is well known that the small cancellation conditions $C'(1/6)$ and $C'(1/4)-T(4)$ imply hyperbolicity (see \cite{GS,Gro}). We will show that a weaker small cancellation hypothesis, which generalizes both $C'(1/6)$ and $C'(1/4)-T(4)$, suffices to prove hyperbolicity of one-relator groups. We include some examples that illustrate the strength and simplicity of this result.

\section{Strictly systolic angled complexes}

We will work with a less rigid notion of simplicial complex. A \textit{quasi-simplicial complex} is a simplicial complex in which we admit multiple edges between vertices. We do not however allow neither loops nor $2$-simplices with two or more edges in common. All the complexes that we deal with are assumed to be locally finite. We say that a quasi-simplicial complex $X$ is \textit{3-flag} if every time $X$ has three faces of a tetrahedron, then the whole tetrahedron is in $X$ (and in particular, the fourth face is in $X$). Note that a flag simplicial complex is, in particular, a $3$-flag quasi-simplicial complex.

Similary as in \cite{Ger, Wi}, we will define a weight function $w$ on the \textit{corners} of the $2$-simplices of $X$. Given a vertex $v\in X$ we denote by $\lk(v)$ the geometric link of $v$ in the $2$-skeleton $X^{(2)}$. The function $w$ assigns a nonnegative value to each edge in $\lk(v)$ for every vertex $v\in X$ (i.e. the $2$-skeleton $X^{(2)}$ is a nonnegative angled $2$-complex in the sense of \cite{Wi}). We require that the image of $w$ is finite and that $w$ satisfies a weak triangular inequality: for any vertex $v$, if $\alpha_{ij}$ is an edge in $\lk(v)$ from $v_i$ to $v_j$ then $w(\alpha_{13})\leq w(\alpha_{12})+w(\alpha_{23})$. The complex $X$ together with a fixed weight function is called an \textit{angled complex}.

Following \cite{HO}, we say that a simple cycle $\sigma$ of length greater than 3 in the link of a vertex $v\in X$ is \textit{$2$-full} if there is no edge in $\lk(v)$ that connects two vertices having a common neighbour in $\sigma$. The angular length of a path in the link of a vertex in an angled complex is the sum of the weights of its edges. An angled complex $X$ is \textit{locally $2\pi$-large} if every $2$-full cycle in every vertex link has angular length $\geq2\pi$.

\begin{definition}
A simply connected, locally $2\pi$-large, $3$-flag, angled complex in which the sum of the internal weights of each triangle is (strictly) less than $\pi$ is called a \textit{strictly systolic angled complex}.
\end{definition}

We will show that strictly systolic angled complexes satisfy a linear isoperimetric inequality. The proof will follow ideas of Gersten \cite{Ger}, Huck and Rosebrock \cite{HR}, Wise \cite{Wi}, Januszkiewickz and \'Swi\c{a}tkowski \cite{JS}, and Huang and Osajda \cite{HO}.

A \textit{singular disk} is a simply connected and planar combinatorial 2-complex whose cells are simplices. Let $X$ be a strictly systolic angled complex, and let $\gamma:S\to X^{(1)}$ be a simplicial map from a triangulation of $S^1$ to the $1$-skeleton of $X$. Its image is a closed edge-path in $X$ which we will also denote by $\gamma$. A \textit{singular diagram} for $\gamma$ is a simplicial map $f:D \to X$ from a singular disk such that $f|_{\partial D} = \gamma$. 

\begin{remark}
Since $X$ is simply connected, every closed edge-path in $X$ admits a singular diagram. This is a direct consequence of the relative simplicial approximation theorem.
\end{remark}

We define the area of a closed edge-path  $\gamma$ in $X$ as
$$\area(\gamma) = \min\{|D| : f:D\to X \text{ is a singular diagram for } \gamma\},$$
where $|D|$ denotes the number of faces ($2$-simplices) in the corresponding singular disk $D$. Let $l(\gamma)$ denote the number of edges in $\gamma$. Our aim is to show that there exists a constant $K>0$ such that $\area(\gamma) \leq Kl(\gamma)$ for every closed edge-path $\gamma$. To prove this we will need well behaved diagrams. A singular diagram is said to be \textit{nondegenerate} if it is injective in every simplex.

\begin{lem}[\cite{JS}, Lemma 1.6]\label{nondegenerate}
Let $\gamma$ be an homotopically trivial closed edge-path in a quasi-simplicial complex $X$. Then there exists a nondegenerate singular diagram for $\gamma$.
\end{lem}

\begin{remark}
Januszkiewickz and \'Swi\c{a}tkowski also showed that the nondegenerate diagram can be taken simplicial. However, we cannot assure this since $X$ is quasi-simplicial.
\end{remark}

By Lemma \ref{nondegenerate} we may assume that the diagrams $f:D\to X$ are nondegenerate. We now adopt a terminology from \cite{HR}. Given a singular diagram $f:D\to X$, we say it is \textit{vertex reduced} if it maps the link of every vertex of $D$ to a path in the link of a vertex in $X$ in which no edge is passed twice in opposite directions. Suppose we have a non vertex reduced diagram $f:D\to X$, with $X$ a strictly systolic angled complex. If we are in the situation where a troublesome link has length 2, the diagram locally looks like in Figure \ref{figone}. We can identify the edges $e_1$ and $e_2$ and collapse the corresponding faces to obtain a new diagram. We call this move an \textit{edge reduction} (see Figure \ref{figone}).

\begin{figure}[h]
     \begin{tikzpicture}[scale=0.5]
    \draw[thick] (0,0) circle (2.5);
    \draw [thick, draw=black, fill=black, fill opacity=0.1] (0,0) circle (2.5);
    \draw[thick] (0,-2.5) -- (0,2.5);
    \filldraw (0,0) circle (2pt);
    \filldraw (0,2.5) circle (2pt);
    \filldraw (0,-2.5) circle (2pt);
    \node[right] at (2.5,0) {$e_1$};
    \node[left] at (-2.5,0) {$e_2$};
    
    \draw[thick,->] (3.5,0) -- (5,0);
    
    \draw[thick] (6,2.5) -- (6,-2.5);
    \filldraw (6,2.5) circle (2pt);
    \filldraw (6,-2.5) circle (2pt);
    \end{tikzpicture}
      \caption{Edge reduction.}
\label{figone}
\end{figure}
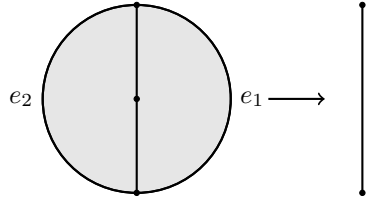

If the link has length greater than 2, we can apply a sequence of diamond moves as in \cite{CH,HR}, followed by an edge reduction. This is shown in Figure \ref{figtwo}. Note that, if the diagram is not vertex reduced, there is a vertex $v$ and two triangles in $D$ which are incident to $v$ that are mapped to the same triangle in $X$. In particular there are two pairs of edges such that each pair is mapped to the same edge. One of these pairs is shown if Figure \ref{figtwo}. For a detailed exposition on diamond moves see \cite{CH}.

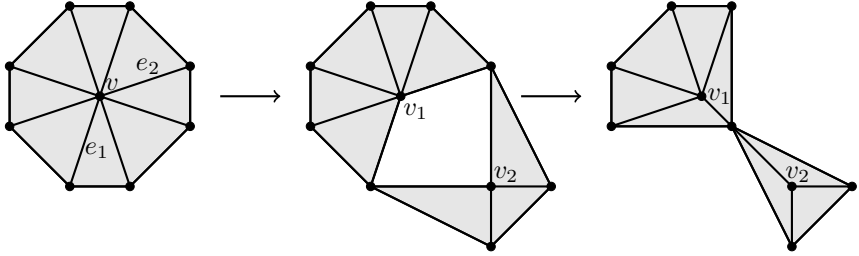
\begin{figure}[h]
     \begin{tikzpicture}[scale=0.8]
    \draw [thick] (0,0) -- (1,0) -- (2,1) -- (2,2) -- (1,3) -- (0,3) -- (-1,2) -- (-1,1) -- cycle;
    \filldraw (0,0) circle (2pt);
    \filldraw (1,0) circle (2pt);
    \filldraw (2,1) circle (2pt);
    \filldraw (2,2) circle (2pt);
    \filldraw (1,3) circle (2pt);
    \filldraw (0,3) circle (2pt);
    \filldraw (-1,2) circle (2pt);
    \filldraw (-1,1) circle (2pt);
    \filldraw (0.5,1.5) circle (2pt);
    \node at (0.7,1.7) {$v$};
    \node at (0.45,0.6) {$e_1$};
    \node at (1.3,2) {$e_2$};
    \draw[thick] (0.5,1.5) -- (0,0);
    \draw[thick] (0.5,1.5) -- (1,0);
    \draw[thick] (0.5,1.5) -- (2,1);
    \draw[thick] (0.5,1.5) -- (2,2);
    \draw[thick] (0.5,1.5) -- (1,3);
    \draw[thick] (0.5,1.5) -- (0,3);
    \draw[thick] (0.5,1.5) -- (-1,2);
    \draw[thick] (0.5,1.5) -- (-1,1);
    \draw [thick, draw=black, fill=black, fill opacity=0.1] (0,0) -- (1,0) -- (2,1) -- (2,2) -- (1,3) -- (0,3) -- (-1,2) -- (-1,1) -- cycle;
    
    \draw[thick,->] (2.5,1.5) -- (3.5,1.5);
    
    \draw [thick] (7,2) -- (6,3) -- (5,3) -- (4,2) -- (4,1) -- (5,0);
    \filldraw (7,2) circle (2pt);
    \filldraw (6,3) circle (2pt);
    \filldraw (5,3) circle (2pt);
    \filldraw (4,2) circle (2pt);
    \filldraw (4,1) circle (2pt);
    \filldraw (5,0) circle (2pt);
    \filldraw (5.5,1.5) circle (2pt);
    \node at (5.75,1.25) {$v_1$};
    \node at (7.25,0.2) {$v_2$};
    \filldraw (7,0) circle (2pt);
    \draw[thick] (5,0) -- (5.5,1.5) -- (7,2);
    \draw[thick] (5,0) -- (7,0) -- (7,2);
    \draw[thick] (5.5,1.5) -- (6,3);
    \draw[thick] (5.5,1.5) -- (5,3);
    \draw[thick] (5.5,1.5) -- (4,2);
    \draw[thick] (5.5,1.5) -- (4,1);
    \filldraw (8,0) circle (2pt);
    \filldraw (7,-1) circle (2pt);
    \draw[thick] (8,0) -- (7,2);
    \draw[thick] (7,-1) -- (5,0);
    \draw[thick] (7,-1) -- (8,0);
    \draw[thick] (7,0) -- (7,-1);
    \draw[thick] (7,0) -- (8,0);
    \draw [thick, draw=black, fill=black, fill opacity=0.1] (7,2) -- (6,3) -- (5,3) -- (4,2) -- (4,1) -- (5,0) -- (5.5,1.5) -- cycle;
    \draw [thick, draw=black, fill=black, fill opacity=0.1] (7,2) -- (8,0) -- (7,-1) -- (5,0) -- (7,0) -- cycle;
    
    \draw[thick,->] (7.5,1.5) -- (8.5,1.5);
    
    \draw [thick] (11,1) -- (11,3) -- (10,3) -- (9,2) -- (9,1) -- (11,1);
    \filldraw (11,1) circle (2pt);
    \filldraw (11,3) circle (2pt);
    \filldraw (10,3) circle (2pt);
    \filldraw (9,2) circle (2pt);
    \filldraw (9,1) circle (2pt);
    \filldraw (10.5,1.5) circle (2pt);
    \node at (10.8,1.5) {$v_1$};
    \node at (12.1,0.2) {$v_2$};
    \filldraw (12,0) circle (2pt);
    \draw[thick] (10.5,1.5) -- (11,3);
    \draw[thick] (10.5,1.5) -- (10,3);
    \draw[thick] (10.5,1.5) -- (9,2);
    \draw[thick] (10.5,1.5) -- (9,1);
    \draw[thick] (10.5,1.5) -- (11,1);
    \filldraw (13,0) circle (2pt);
    \filldraw (12,-1) circle (2pt);
    \draw[thick] (13,0) -- (11,1);
    \draw[thick] (12,-1) -- (11,1);
    \draw[thick] (12,-1) -- (13,0);
    \draw[thick] (12,0) -- (12,-1);
    \draw[thick] (12,0) -- (13,0);
    \draw[thick] (12,0) -- (11,1);
    \draw [thick, draw=black, fill=black, fill opacity=0.1] (11,1) -- (11,3) -- (10,3) -- (9,2) -- (9,1) -- (11,1);
    \draw [thick, draw=black, fill=black, fill opacity=0.1] (11,1) -- (13,0) -- (12,-1) -- (11,1);
    \end{tikzpicture}
      \caption{Diamond move. Edges $e_1$ and $e_2$ are mapped to the same edge.}
\label{figtwo}
\end{figure}

Note that these moves reduce the number of faces of the diagram. Therefore, starting with a nondegenerate singular diagram, we can obtain a nondegenerate and vertex reduced singular diagram with the same boundary. Observe that in such a diagram the image of the link of every vertex can be decomposed in simple paths. In particular, the image of the link of every interior vertex can be decomposed in simple cycles.

\begin{lem}\label{triangulated}
Let $X$ be a strictly systolic angled complex, $v$ a vertex of $X$ and $\sigma$ a simple cycle in $\lk(v)$. Then $\sum_{c\in\sigma}w(c)\geq 2\pi$ or $\sigma$ is the boundary of a triangulated disk without interior vertices, whose edges are in $\lk(v)$.
\end{lem}

\begin{proof}
We proceed by induction on the length of $\sigma$. There are no cycles of length 2, because $X$ is quasi-simplicial. If it has length 3, then the claim trivially holds.
Suppose $\sigma$ has more than 3 edges. If it is $2$-full, then it has angular length $\geq 2\pi$. If it is not $2$-full, then there exists an edge $e$ in $\lk(v)$ that connects two vertices of $\sigma$ with a common neighbour. This edge subdivides $\sigma$ in two paths: one of length 2, and another one of length $l(\sigma)-2$, which we call $\sigma_1$ and $\sigma_2$ respectively (see Figure \ref{figthree}). By inductive hypothesis, $\sigma_2\cup e$ either subdivides into triangles or has angular length $\geq 2\pi$. If it subdivides, it induces a subdivision for $\sigma$. In the other case, by the triangle inequality we have
\[
2\pi \leq \sum_{c\in \sigma_2\cup e}\omega(c) \leq \sum_{c\in \sigma_2\cup\sigma_1}\omega(c) = \sum_{c\in \sigma}\omega(c).
 \qedhere
 \] 
\end{proof}
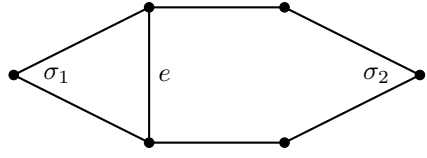
\begin{figure}[h]
    \begin{tikzpicture}[scale=0.9]
    \draw[thick] (0,0) -- (2,-1) -- (4,-1) -- (6,0) -- (4,1) -- (2,1) --cycle;
    \filldraw (0,0) circle (2pt);
    \filldraw (2,-1) circle (2pt);
    \filldraw (4,-1) circle (2pt);
    \filldraw (6,0) circle (2pt);
    \filldraw (4,1) circle (2pt);
    \filldraw (2,1) circle (2pt);
    \draw[thick] (2,1) -- (2,-1);
    \node[right] at (2,0) {$e$};
    \node[right] at (0.3,0) {$\sigma_1$};
    \node[left] at (5.7,0) {$\sigma_2$};
    \end{tikzpicture}  
    \caption{Subdividing $\sigma$.}
\label{figthree}
\end{figure}

\begin{lem}\label{reduced}
Let $X$ be a strictly systolic angled complex and $f:D\to X$ a nondegenerate and vertex reduced singular diagram for a closed edge-path $\gamma$. Then there exists a nondegenerate and vertex reduced singular diagram $g:D'\to X$ for $\gamma$ such that the image by $g$ of the links of the interior vertices of $D'$ can be decomposed in simple cycles of angular length $\geq2\pi$. 
\end{lem}

\begin{proof}
Let $v$ be an interior vertex of $D$ such that one of the simple cycles in the decomposition of $f(\lk(v))$, say $\sigma$, has angular length $<2\pi$. If $\sigma$ is not the only simple cycle in the decomposition of $f(\lk(v))$, then there exist two edges $e_1$ and $e_2$ incident to $v$ satisfying $f(e_1)=f(e_2)$. By a diamond move, we can obtain a new nondegenerate singular diagram $f'$ for $\gamma$ with the same number of faces and new vertices $v_1$ and $v_2$ (see Figure \ref{figtwo}). Therefore we can assume that $\sigma$ is the only simple cycle in the decomposition of $f(\lk(v_1))$. Since its angular length is $<2\pi$, by Lemma \ref{triangulated} it subdivides in triangles in $\lk(f(v_1))$. This corresponds with the situation shown in Figure \ref{figfour}. Since $X$ is $3$-flag, we can modify $f$ removing the troublesome vertex $v_1$, as shown in Figure \ref{figfive}.

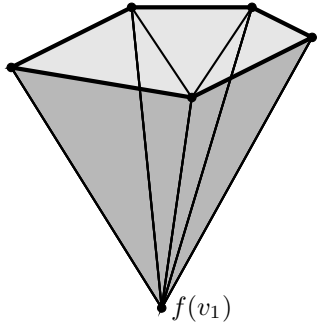
\begin{figure}[h]
\begin{tikzpicture}[scale=0.8]
\draw[ultra thick] (0,0) -- (3,-1/2) -- (5,1/2) -- (4,1) -- (2,1) -- cycle;
\draw[thick] (2,1) -- (3,-1/2) -- (4,1);
\filldraw (3,-1/2) circle (2pt);
\filldraw (5,1/2) circle (2pt);
\filldraw (4,1) circle (2pt);
\filldraw (2,1) circle (2pt);
\filldraw (0,0) circle (2pt);
\filldraw (5/2,-4) circle (2pt);
\draw[thick] (5/2,-4) -- (0,0);
\draw[thick] (5/2,-4) -- (3,-1/2);
\draw[thick] (5/2,-4) -- (5,1/2);
\draw[thick] (5/2,-4) -- (4,1);
\draw[thick] (5/2,-4) -- (2,1);
\draw [thick, draw=black, fill=black, fill opacity=0.2] (0,0) -- (5/2,-4) -- (3,-1/2) -- cycle;
\draw [thick, draw=black, fill=black, fill opacity=0.2] (5,1/2) -- (5/2,-4) -- (3,-1/2) -- cycle;
\draw [thick, draw=black, fill=black, fill opacity=0.1] (4,1) -- (5/2,-4) -- (5,1/2) -- cycle;
\draw [thick, draw=black, fill=black, fill opacity=0.1] (2,1) -- (5/2,-4) -- (4,1) -- cycle;
\draw [thick, draw=black, fill=black, fill opacity=0.1] (2,1) -- (5/2,-4) -- (0,0) -- cycle;
\node[right] at (5/2,-4) {$f(v_1)$};
\end{tikzpicture}
\caption{The cycle $\sigma$ subdivides in $\lk(f(v_1))$.}
\label{figfour}
\end{figure}

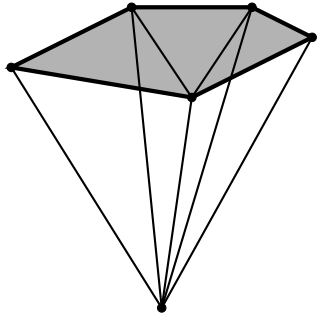
\begin{figure}[h]
\begin{tikzpicture}[scale=0.8]
\draw[ultra thick] (0,0) -- (3,-1/2) -- (5,1/2) -- (4,1) -- (2,1) -- cycle;
\draw[thick] (2,1) -- (3,-1/2) -- (4,1);
\filldraw (3,-1/2) circle (2pt);
\filldraw (5,1/2) circle (2pt);
\filldraw (4,1) circle (2pt);
\filldraw (2,1) circle (2pt);
\filldraw (0,0) circle (2pt);
\filldraw (5/2,-4) circle (2pt);
\draw[thick] (5/2,-4) -- (0,0);
\draw[thick] (5/2,-4) -- (3,-1/2);
\draw[thick] (5/2,-4) -- (5,1/2);
\draw[thick] (5/2,-4) -- (4,1);
\draw[thick] (5/2,-4) -- (2,1);
\draw [thick, draw=black, fill=black, fill opacity=0.3] (0,0) -- (3,-1/2) -- (5,1/2) -- (4,1) -- (2,1) -- cycle;
\end{tikzpicture}
\caption{Removal of $v_1$.}
\label{figfive}
\end{figure}

After applying this change, we obtain a new nondegenerate singular diagram for $\gamma$ with less faces. Then we can make it vertex reduced by reducing the number of faces once again, and continue with this process. Since the number of faces decreases at each step, the process stops and we obtain a nondegenerate and vertex reduced singular diagram $g:D'\to X$ for $\gamma$, which satisfies the desired conditions.
\end{proof}

 If $X$ is a strictly systolic angled complex and $f:D\to X$ is a singular diagram with $f$ nondegenerate, we can pullback the weights of the corners of $X$ to $D$. We will apply the combinatorial Gauss-Bonnet theorem to the angled $2$-complex $D$. We recall first  some combinatorial notions of curvature from \cite{Wi}.

\begin{definition}
Let $L$ be an angled $2$-complex whose cells are simplices. If $v$ is a vertex of $L$, the curvature of $v$ is defined as $\kappa(v) = 2\pi - \pi\chi(\lk(v)) - (\sum_{c\in v}\omega(c))$, where the sum is taken over all the corners at $v$, and $\chi$ denotes the Euler characteristic. The curvature of a face ($2$-simplex) $F$ is
$\kappa(F) = (\sum_{c\in F}\omega(c)) - \pi$, where the sum is taken over all corners in $F$.
\end{definition}

\begin{thm}[Combinatorial Gauss-Bonnet Theorem (\cite{BB,Wi})]
Let $L$ be an angled $2$-complex. Then
$$\sum_{F\in \text{faces}(L)}\kappa(F) + \sum_{v\in L^{(0)}}\kappa(v) = 2\pi\chi(L).$$
\end{thm}

Now we are ready to prove the linear isoperimetric inequality.

\begin{thm}\label{linear}
Let $X$ be a strictly systolic angled complex. Then there exists a constant $K>0$ such that
$$\area(\gamma) \leq Kl(\gamma),$$
for every closed edge-path $\gamma$ in $X$.
\end{thm}
\begin{proof}
We will find a positive constant $K$ such that for any closed edge-path $\gamma$, there exists a singular diagram $g:D\to X$ for $\gamma$ with $|D|\leq Kl(\gamma)$. Given $\gamma$, take a nondegenerate and vertex reduced singular diagram $g:D\to X$ satisfying the conditions of Lemma \ref{reduced}. Since $g$ is nondegenerate, we can pullback $w$ to $D$ via $g$. Since the sum of the internal weights of each face of $X$ is $<\pi$,  then $\kappa(F)<0$ for every face $F$ of $D$. Furthermore, since the image of $w$ is finite, $\kappa(F)\leq M < 0$ for a certain negative constant $M$ independent from $g,D$ and $\gamma$. The image of the link of each interior vertex can be decomposed into simple cycles of angular length $\geq2\pi$. Then $\kappa(v)\leq0$ if $v$ is an interior vertex of $D$. Now we apply the combinatorial Gauss-Bonnet theorem to $D$ and we get
$$M|D|\geq \sum_{F\in\text{faces}(D)}\kappa(F) = 2\pi\chi(D) - \sum_{v\in D^{(0)}}\kappa(v) = 2\pi - \sum_{v\in D^{(0)}}\kappa(v),$$
and therefore
$$|D| \leq \frac{1}{-M}(\sum_{v\in D^{(0)}}\kappa(v)-2\pi) \leq \frac{1}{-M}(\sum_{v\in \partial D^{(0)}}\kappa(v)-2\pi),$$
where $\partial D$ denotes the boundary of $D$. Note that the number of vertices in $\partial D$ is less than or equal to $l(\gamma)$. Since the links of the vertices in $\partial D$ have nonnegative Euler characteristic, and since the weight is nonnegative, their curvature is at most $2\pi$. Setting  $K=\frac{2\pi}{-M}$, we obtain $|D|\leq K l(\gamma)$.
\end{proof}

In the light of Theorem \ref{linear}, and by \cite[III.2.9]{BH}, we obtain the following corollaries.

\begin{cor}\label{maincor1}
The 1-skeleton $X^{(1)}$ of a strictly systolic angled complex $X$ with its standard geodesic metric is hyperbolic. More generally, if we endow $X$ with a piecewise Euclidean metric with $Shapes(X)$ finite, then $X$ is hyperbolic.
\end{cor}

Recall that $Shapes(X)$ is the set of isometry classes of the simplices of $X$ (see \cite{BH}). 

A group $\Gamma$ which acts properly and cocompactly by simplicial automorphisms on a strictly systolic angled complex is called a \textit{strictly systolic} group. Note that, similarly as in \cite[Theorem 3.1]{HO}, since the class of locally $2\pi$-large, $3$-flag angled complexes is closed under taking full subcomplexes and covers, by \cite[Theorem 1.1]{HP} finitely presented groups of strictly systolic groups are strictly systolic. From Corollay \ref{maincor1} we obtain the following result.
 
\begin{cor}\label{maincor2}
Strictly systolic groups are hyperbolic. Moreover, any finitely presented group of a strictly systolic group is strictly sytolic, and hence, hyperbolic.
\end{cor}

Note that in Corollary \ref{maincor1}, the angles (or weights) of the strictly systolic angled complex $X$ are independent of the metric. However, if we are given a simplicial or quasi-simplicial complex $X$ with a piecewise Euclidean metric on the $2$-skeleton $X^{(2)}$ with a certain link condition, we can define a weight function that makes $X$ a strictly systolic angled complex. The following definition is analogous to the notion of \textit{metrically systolic complex} introduced by Huang and Osajda \cite{HO}. It is, in some sense, the metric counterpart to the $7$-systolic simplicial complexes of \cite{JS}.

\begin{definition}
A \textit{metrically strictly systolic} complex is a simply connected and $3$-flag quasi-simplicial complex $X$  such that its $2$-skeleton is equipped with a piecewise Eucliden metric (with finite shapes) and such that the angular distance induced in the links of the vertices satisfies the weak triangular inequality and the vertex links (with the angular distance) are strictly $2\pi$-large, i.e. every $2$-full cycle has angular length $>2\pi$.
\end{definition}

\begin{prop}
Metrically strictly systolic complexes are strictly systolic angled complexes.
\end{prop}

\begin{proof}
Since $X^{(2)}$ has finite shapes, there exists $L>2\pi$ such that every $2$-full cycle has angular length $\geq L$. Then we can define an appropriate weight in the corners by subtracting a fixed small enough $\delta>0 $ from every angle.
\end{proof}

Note that Corollaries \ref{maincor1} and \ref{maincor2} generalize Januszkiewickz and \'Swi\c{a}tkowski's results on $7$-systolic simplicial complexes and groups \cite[Theorem 2.1 and Corollary 2.2]{JS} since any $7$-systolic simplicial complex with the standard Euclidean metric on its $2$-skeleton (with all triangles equilateral) is metrically strictly systolic.

\section{Hyperbolicity of one-relator groups}

 It is well known that the groups which admit finite presentations satisfying the metric small cancellation condition $C'(1/6)$ are hyperbolic \cite{Gro}. Recall that a finite presentation satisfies the $C'(\lambda)$ condition (with $0<\lambda<1$) if every piece $W$ in a relator $R$ has $l(W)< \lambda l(R)$. Here $l(W)$ denotes the length of the word $W$. Also the conditions $C'(1/4) - T(4)$ imply hyperbolicity \cite{GS}. The main result of this section generalizes both results for one-relator groups. Since all one-relator groups with torsion are hyperbolic, we will consider only one-relator groups without torsion. They are given by presentations $P=\langle \mathcal{A}\ |\ R  \rangle$, where $\mathcal{A}$ is finite and $R$ is a cyclically reduced word which is not a proper power. Note that no proper subword of $R$ is trivial in the presented group $\Gamma$ (see \cite{Wei}).
 
In the case of one-relator presentations, the condition $T(4)$ can be restated as follows: the cyclically reduced word $R$ does not contain pieces $W_1,W_2,W_3$ such that $W_1W_2$, $W_1W_3$ and $W_2^{-1}W_3$ are nonempty subwords of $R$ or $R^{-1}$ or any cyclic permutation of them. 

We will apply Corollary \ref{maincor2} to prove that  $C'(1/4)$ together with a much weaker condition than $T(4)$ guarantees hyperbolicity of one-relator groups. Contrary to condition $T(4)$, we allow the existence of pieces $W_1,W_2,W_3$ such that $W_1W_2$, $W_1W_3$ and $W_2^{-1}W_3$ are nonempty subwords, but we impose a condition on their lengths. Concretely,  condition $T(4)$ is replaced by the following weaker \textit{Condition (T')}. 

\textit{Condition (T'):}  If there exist pieces $W_1,W_2,W_3$ of $R$ such that $W_1W_2$, $W_1W_3$ and $W_2^{-1}W_3$ are nonempty subwords of $R$ or $R^{-1}$ or any cyclic permutation of them, then $l(W_1)+l(W_2)+l(W_3)<l(R)/2$.

The main result of this section is the following.

\begin{thm}\label{onerelator}
Let $\Gamma$ be a one-relator group with presentation $P=\langle \mathcal{A}\ |\ R \rangle$. If $P$ satisfies the metric small cancellation condition $C'(1/4)$ and Condition (T'), then $\Gamma$ is hyperbolic.
\end{thm}

Note that a  $C'(1/6)$ one-relation presentation automatically satisfies Condition (T').

To prove this theorem, we will construct a strictly systolic angled complex from $P$, on which $\Gamma$ acts properly and cocompactly by simplicial automorphisms. This construction is inspired in Huang and Osajda's construction for Artin groups \cite{HO} but it is adapted to the geometry of these groups.

Before we proceed with the proof, we illustrate the result with an example.

\begin{example}
Consider the presentation $P=\langle a,b | a^4babab^{-1}a^{-1}b^3a^{-1}b \rangle$. The presented group satisfies the hypotheses of the theorem and, hence, it is hyperbolic. Note that it is neither $C'(1/6)$ nor $T(4)$. Also it does not satisfy any hyperbolic weight test \cite{Ger, HR} and it is not in any of the families classified by Ivanov and Schupp \cite{IS}.
\end{example}

\begin{proof}[Proof of Theorem \ref{onerelator}]

Let $r=l(R)$. We can suppose that $r\geq 4$ since $\Gamma$ is clearly hyperbolic when $r\leq 3$.
Let $K_P$ be the standard $2$-complex associated to $P$. Recall that $K_P$ has one $0$-cell, one $1$-cell for each generator $a\in \mathcal{A}$ and one $2$-cell corresponding to the word $R$. We denote by $\tilde{K}_P$ its universal cover. Following the terminology of \cite{HO}, the $2$-cells of $\tilde{K}_P$ (corresponding to all the liftings of the unique $2$-cell of $K_P$) will be called precells. We triangulate each precell of $\tilde{K}_P$ by adding a central vertex. 

Now, if two precells intersect, they do so in a disjoint union of vertices and paths. This is because no proper subword of $R$ is trivial in $\Gamma$. Notice that each intersection, when it is not a single vertex, amounts to a piece in $R$. Let $C_1$ and $C_2$ be two intersecting precells with corresponding centers $c_1$ and $c_2$. For each connected component of their intersection we add an edge between $c_1$ and $c_2$. Note that there could be more than one edge between two centers. Let $v_1,...,v_k$ be the vertices of one component of the intersection. Then for each $i$ we also add triangles with vertices $\{c_1,c_2,v_i\}$ (one of the edges of the boundary of the triangle is the corresponding edge between the centers). We fill the necessary tetrahedrons for the complex to be $3$-flag (see Figure \ref{figsix}).

\begin{figure}[h]
   \begin{tikzpicture}[scale=0.9]
    \draw[thick] (0,0) -- (0,-1) -- (1,-2) -- (2,-2) -- (3,-2) -- (3,1) -- (2,1) -- (1,1) -- (0,0);
    \draw[thick] (3,1) -- (3,-2) -- (4,-2) -- (5,-2) -- (6,-1) -- (6,0) -- (5,1) -- (4,1) -- (3,1);
    \filldraw (1.5,-0.5) circle (2pt);
    \filldraw (4.5,-0.5) circle (2pt);
    \filldraw (0,0) circle (2pt);
    \filldraw (0,-1) circle (2pt);
    \filldraw (1,-2) circle (2pt);
    \filldraw (2,-2) circle (2pt);
    \filldraw (3,-2) circle (2pt);
    \filldraw (3,1) circle (2pt);
    \filldraw (2,1) circle (2pt);
    \filldraw (1,1) circle (2pt);
    \filldraw (4,-2) circle (2pt);
    \filldraw (5,-2) circle (2pt);
    \filldraw (6,-1) circle (2pt);
    \filldraw (6,0) circle (2pt);
    \filldraw (5,1) circle (2pt);
    \filldraw (4,1) circle (2pt);
    \filldraw (3,0) circle (2pt);
    \filldraw (3,-1) circle (2pt);
    
    \draw[thick] (1.5,-0.5) -- (0,0);
    \draw[thick] (1.5,-0.5) -- (0,-1);
    \draw[thick] (1.5,-0.5) -- (1,-2);
    \draw[thick] (1.5,-0.5) -- (2,-2);
    \draw[thick] (1.5,-0.5) -- (3,-1);
    \draw[thick] (1.5,-0.5) -- (3,0);
    \draw[thick] (1.5,-0.5) -- (2,1);
    \draw[thick] (1.5,-0.5) -- (1,1);
    \draw[thick] (1.5,-0.5) -- (3,1);
    \draw[thick] (1.5,-0.5) -- (3,-2);
    \draw[thick] (4.5,-0.5) -- (3,0);
    \draw[thick] (4.5,-0.5) -- (3,-1);
    \draw[thick] (4.5,-0.5) -- (4,-2);
    \draw[thick] (4.5,-0.5) -- (5,-2);
    \draw[thick] (4.5,-0.5) -- (6,-1);
    \draw[thick] (4.5,-0.5) -- (6,0);
    \draw[thick] (4.5,-0.5) -- (5,1);
    \draw[thick] (4.5,-0.5) -- (4,1);
    \draw[thick] (4.5,-0.5) -- (3,1);
    \draw[thick] (4.5,-0.5) -- (3,-2);
    \draw[thick] (1.5, -0.5) -- (4.5, -0.5);
    \draw [thick, draw=black, fill=black, fill opacity=0.2] (1.5, -0.5) -- (3,0) -- (4.5, -0.5) -- (3,-1) -- cycle;
    \draw [thick, draw=black, fill=black, fill opacity=0.2] (1.5, -0.5) -- (3,1) -- (4.5, -0.5) -- (3,-2) -- cycle;
    
    \node[right] at (0,1) {$C_1$};
    \node[right] at (5.5,1) {$C_2$};
    \node[above right] at (4.5,-0.49) {$c_2$};
    \node[above left] at (1.5,-0.49) {$c_1$};
    \end{tikzpicture}
    \caption{Intersection of two precells.}
    \label{figsix}
\end{figure}
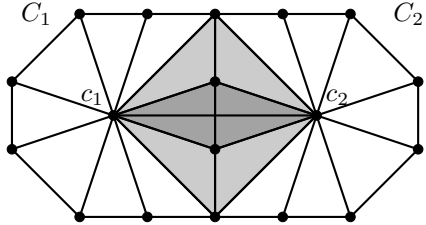

If three precells $C_1, C_2$ and $C_3$ intersect, we add triangles with vertices in the three centers (one for each component of $C_1\cap C_2\cap C_3$) and the necessary tetrahedrons for it to be $3$-flag. We denote by $X$ the complex that we obtained. Note that $X$ is quasi-simplicial since the presentation satisfies, in particular, the metric condition $C'(1/2)$. It is clear that $X$ is simply connected (since $\tilde{K}_P$ is) and it is $3$-flag by construction. Note that $\Gamma$ acts simplicially, properly and cocompactly on $X$ since the modifications that we made on $\tilde{K}_P$ are equivariant.

Now we define a weight function $w$ on $X$ and study under which conditions $(X,w)$ is a strictly systolic angled complex.  We will show below that, by construction of $X$, the links of the original vertices of $X$ do not have $2$-full cycles. Therefore we only need to control the weights of the corners at the central vertices and the sum of the internal angles of the triangles. 

There are three kinds of triangles in $X$:

\begin{enumerate}\itemsep0.3em
    \item triangles from the subdivision of a precell,
    \item triangles with two central vertices,
    \item triangles with three central vertices.
\end{enumerate}

Recall that $r$ denotes the length of the relator $R$. If the triangle is of type (1), then the weight assigned to the angle of the central vertex is $\frac{2\pi}{r}$, and the weights of the other two are $0$.

If the triangle is of type (2), the weight of the two central angles is $\frac{l}{r}\pi$, where $l$ is the length of the component of the intersection corresponding to that triangle. The remaining angle equals $0$.

If the triangle is of type (3), it corresponds to the intersection of three precells $C_1,C_2$ and $C_3$ with centers $c_1,c_2$ and $c_3$. Let $l_{12}$ be the length of the component of the intersection between $C_1$ and $C_2$. Analogously we define $l_{13}$, $l_{23}$ and $l_{123}$ (the length of the component of $C_1\cap C_2\cap C_3$). The weights at the angles at $c_1,c_2$ and $c_3$ are $\frac 1r.(l_{12}+l_{13}-2l_{123})\pi$,  $\frac 1r.(l_{12}+l_{23}-2l_{123})\pi$ and $\frac 1r.(l_{13}+l_{23}-2l_{123})\pi$ respectively. The sum of the angles is $\frac 1r .(2l_{12}+2l_{13}+2l_{23}-6l_{123})\pi$, which is less than $\pi$ if the presentation is $C'(1/4)$ and if, whenever the intersection of the three cells is a vertex, $l_{12}+l_{13}+l_{23}<\frac{r}{2}$, which is Condition (T').

Under these circumstances, all the triangles have inner weight less than $\pi$, and the triangle inequality is easily seen to be satisfied. 

Now we analyze the links of the vertices in $X$ and their $2$-full cycles. There are two types of vertices in $X$. The original vertices (vertices of $\tilde{K}_P$) and the central vertices. Note that the links of any two vertices of the same kind are equal, so we need to verify only two cases.

First we study the links of central vertices (see Figure \ref{figseven}). Let $c$ be a central vertex in a precell $C$. Its link has two kinds of vertices:
\begin{enumerate}\itemsep0.3em
    \item[(i)] those coming form edges of the triangulation of $C$,
    \item[(ii)] those corresponding to edges between $c$ and another central vertex.
\end{enumerate}

\begin{figure}[h]
\scalebox{0.8}{
   \begin{tikzpicture}
    \pgfmathsetmacro\myradi{12/2/pi}
    \draw[decoration={markings,
    mark=at position 1cm with \coordinate (c1);,
    mark=at position 2cm with \coordinate (c2);,
    mark=at position 3cm with \coordinate (c3);,
    mark=at position 4cm with \coordinate (c4);,
    mark=at position 5cm with \coordinate (c5);,
    mark=at position 6cm with \coordinate (c6);,
    mark=at position 7cm with \coordinate (c7);,
    mark=at position 8cm with \coordinate (c8);,
    mark=at position 9cm with \coordinate (c9);,
    mark=at position 10cm with \coordinate (c10);,
    mark=at position 11cm with \coordinate (c11);,
    mark=at position 12cm with \coordinate (c12);}, postaction={decorate}] 
    circle [thick, radius=\myradi cm];
    
    \filldraw (c1) circle (2pt);
    \filldraw (c2) circle (2pt);
    \filldraw (c3) circle (2pt);
    \filldraw (c4) circle (2pt);
    \filldraw (c5) circle (2pt);
    \filldraw (c6) circle (2pt);
    \filldraw (c7) circle (2pt);
    \filldraw (c8) circle (2pt);
    \filldraw (c9) circle (2pt);
    \filldraw (c10) circle (2pt);
    \filldraw (c11) circle (2pt);
    \filldraw (c12) circle (2pt);
    
    \filldraw (1,0) circle (2pt);
    \filldraw (0,1) circle (2pt);
    
    \draw[thick] (1,0) -- (c1);
    \draw[thick] (1,0) -- (c2);
    \draw[thick] (1,0) -- (c12);
    
    \draw[thick] (0,1) -- (c4);
    \draw[thick] (0,1) -- (c3);
    \draw[thick] (0,1) -- (c1);
    \draw[thick] (0,1) -- (c2);
    
    \draw[thick] (1,0) -- (0,1);
    \end{tikzpicture}}
\caption{Link of a central vertex. Vertices of type (i) are located in the exterior circle, and vertices of type (ii) in the interior.}
\label{figseven}
\end{figure}
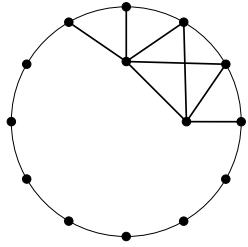

Let $\sigma$ be a $2$-full cycle in $\lk(c)$. We will show that its angular length is $\geq2\pi$. We examine the possible cases.

\textbf{Case 1} If $\sigma$ only passes through vertices of type (i), then $\sigma$ is a circle of angular length $2\pi$.

\textbf{Case 2} Suppose $\sigma$ only passes through vertices of type (ii). Those vertices correspond to paths in the boundary of the precell. Each path intersects another two paths (because $\sigma$ is a cycle) and no three paths have common intersection (because $\sigma$ is $2$-full). Therefore their union covers the boundary of the precell. Let $s_1,...,s_k$ be those paths. Then the angular length of $\sigma$ equals 
$$\frac{(l(s_1)+l(s_2)-l(s_1\cap s_2))+...+(l(s_k)+l(s_1)-l(s_k\cap s_1))}{r}\pi.$$
By an inclusion exclusion argument this is exactly $2\pi$.

\textbf{Case 3} Finally, suppose $\sigma$ passes through vertices of both kinds. Given an orientation for $\sigma$, let $v_1$ be a vertex of type (i) such that the following vertex in $\sigma$, say $u$, is of type (ii). Let $v_2$ be the next vertex of type (i) that appears (note that $v_1\neq v_2$). The vertex $u$ is adjacent to some vertices $v'_1,...,v'_l$ between $v_1$ and $v_2$. Denote by $[v_1,v_2]$ the edge-path passing only through vertices of type (i) that connects $v_1$ with $v_2$ and which contains $v'_1,...,v'_l$. By an argument analogous to that of case 2, the segment of $\sigma$ that goes from $v_1$ to $v_2$ has angular length $\geq \frac{2 l([v_1,v_2])}{r}\pi$. Inductively, the remaining segment of $\sigma$ has angular length $\geq \frac{2(r-l([v_1,v_2]))}{r}\pi$. Therefore $\sigma$ has angular length $\geq 2\pi$.

Finally we analyze the links of the original vertices. By the definition of $X$, the link of any original vertex in the $2$-skeleton of $X$ is obtained from the link of the vertex in the universal cover $\tilde{K}_P$ by subdividing all the edges (adding a new vertex in every edge of the original link), and then adding all the edges between all pairs of new vertices of the link. In particular, the links of the original vertices in $X^{(2)}$ do not have $2$-full cycles, and therefore the link condition around these vertices is automatically satisfied.
\end{proof}

Note that, by Corollary \ref{maincor2}, all finitely presented subgroups of the one-relator groups $\Gamma$ of Theorem \ref{onerelator} are also hyperbolic. But this fact can also be  deduced from Theorem \ref{onerelator} by a well known result of Gersten on finitely presented subgroups of hyperbolic groups in dimension $2$ \cite{Ger2}. Recall that one-relator groups without torsion have cohomological dimension $2$.

\begin{example}
The construction that we introduced in the proof of Theorem \ref{onerelator} can also be applied to prove hyperbolicity of one-relator groups which do not satisfy the hypotheses of the theorem. One of such examples is the group presented by $\langle a,t | at^{-1}ata^2t^{-2}a^{-1}t^2 \rangle$. This group appeared in \cite{Kap}. In \cite[Theorem A]{Kap} this group was proved to be hyperbolic with very different techniques.
\end{example}


\begin{thebibliography}{99}

\bibitem{BB} Ballmann, W. and Buyalo, S. \textit{Nonpositively curved metrics on 2-polyhedra}. Math. Z. \textbf{222} (1996), no. 1, 97--134.
\bibitem{BH} Bridson, M. and Haefliger, A. \textit{Metric spaces of non-positive curvature}. Springer Verlag (1999).
\bibitem{CH} Collins, D. J. and Huebschmann, J. \textit{Spherical diagrams and identities among relations}. Math. Ann. \textbf{261} (1982), 155--183.
\bibitem{GW} Gardam, G. and Woodhouse, D. \textit{The geometry of one-relator groups satisfying a polynomial isoperimetric inequality}. Proc. Amer. Math. Soc. \textbf{147} (2019) 125--129 
\bibitem{Ger} Gersten, S.M. \textit{Reducible diagrams and equations over groups}. Essays in group theory, Springer-Verlag, 1987. 15--73.
\bibitem{Ger2} Gersten, S.M. \textit{Subgroups of word hyperbolic groups in dimension 2}. J. Lond. Math. Soc. \textbf{54} (1996) 261--283.
\bibitem{GS} Gersten, S.M. and Short, H. \textit{Small cancellation theory and automatic groups}. Invent. Math. \textbf{102} (1990) 305--334.
\bibitem{Gro} Gromov, M. \textit{Hyperbolic groups}. Essays in group theory, Springer-Verlag, 1987. 75--263.
\bibitem{HP} Hanlon, R. and Mart\'inez-Pedroza, E. \textit{Lifting group actions,
equivariant towers and subgroups of non-positively curved groups}. Algebr. Geom.
Topol. \textbf{14} (2014) 2783--2808. 
\bibitem{HO} Huang, J. and Osajda, D. \textit{Simplicial nonpositive curvature}. Math. Ann. (2019), in press.
\bibitem{HR} Huck, G. and Rosebrock, S. \textit{Weight tests and hyperbolic groups}. London Math. Soc. Lecture Note Ser. \textbf{204} (1995), 174--186.
\bibitem{IS} Ivanov, S.V. and Schupp, P.E. \textit{On the hyperbolicity of small cancellation groups and one-relator groups}. Trans. Amer. Math. Soc. \textbf{350} (1998), no. 5, 1851--1894.
\bibitem{JS} Januszkiewickz, T. and \'Swi\c{a}tkowski, J. \textit{Simplicial nonpositive curvature}. Publ. Math. Inst. Hautes Etudes Sci. \textbf{104} (2006), 1--85.
\bibitem{Kap} Kapovich, I. \textit{Howson property and one-relator groups}. Comm. Algebra \textbf{27} (1999), no. 3, 1057--1072.
\bibitem{KMS} Karrass, J. and Magnus, W. and Solitar, D. \textit{Combinatorial group theory}. Interscience Pub., John Wiley and Sons, 1955.
\bibitem{New} Newman, B.B. \textit{Some results on one-relator groups}, Bull. Amer. Math. Soc. \textbf{74} (1968), no. 3, 568--571. 
\bibitem{Wei} Weinbaum, C.M. \textit{On relators and diagrams for groups with one defining relation}. Illinois J. Math. \textbf{16} (1972), no. 2, 308--322.
\bibitem{Wi} Wise, D. \textit{Sectional curvature, compact cores and local quasiconvexity}. Geom. Funct. Anal. \textbf{14} (2004), 433--468.

\end{thebibliography}
\end{document}